\documentclass[11pt]{article}
\usepackage[margin=1in]{geometry}

\usepackage{paralist, amsmath, amsthm, amssymb, color, graphicx}
\usepackage{hyperref}

\theoremstyle{definition}
\newcounter{thm}
\newtheorem{theorem}[thm]{Theorem}
\newtheorem{proposition}[thm]{Proposition}
\newtheorem{lemma}[thm]{Lemma}
\newtheorem{corollary}[thm]{Corollary}

\newtheorem{definition}[thm]{Definition}

\newtheorem{rmk}[thm]{Remark}

\newcommand{\al}{\alpha}
\newcommand{\be}{\beta}
\newcommand{\ga}{\gamma}
\newcommand{\de}{\delta}

\newcommand{\la}{\lambda}
\newcommand{\om}{\omega}

\DeclareMathOperator{\uc}{excess}

\newcommand{\ppi}{\pi}
\newcommand{\psig}{\sigma}

\newcommand{\nsig}{\#\mS}

\newcommand{\dd}{d}

\newcommand{\xx}{\textbf{x}}
\newcommand{\mA}{\mathcal{A}}\newcommand{\hmA}{\hat{\mA}}
\newcommand{\mB}{\mathcal{B}}
\newcommand{\mC}{\mathcal{C}}

\newcommand{\mM}{\mathcal{M}}

\newcommand{\mS}{\mathcal{S}}\newcommand{\hmS}{\hat{\mS}}
\newcommand{\mT}{\mathcal{T}}
\newcommand{\mTR}{\widetilde{\mathcal{T}}}

\newcommand{\GG}{G}
\newcommand{\PP}{\mathbb{P}}
\newcommand{\EE}{\mathbb{E}}

\newcommand{\fig}[3]{\begin{figure}[h]\begin{center}\includegraphics[#1]{#2}\end{center}\caption{#3}\label{fig:#2}\end{figure}}

\title{Separation probabilities for products of permutations}
 
\author{Olivier Bernardi\thanks{O.B. aknowledges support from NSF grant DMS-1068626, ANR A3, and ERC Explore-Maps.}, Rosena R. X. Du, Alejandro H. Morales and Richard P. Stanley}

\date{\today}

\begin{document}
\setcounter{tocdepth}{2}

\maketitle

\begin{abstract}
We study the mixing properties of permutations obtained as a product of two uniformly random permutations of fixed cycle types. For instance, we give an exact formula for the probability that elements $1,2,\ldots,k$ are in distinct cycles of the random permutation of $\{1,2,\ldots,n\}$ obtained as product of two uniformly random $n$-cycles.
\end{abstract}

\section{Introduction}\label{sec:intro}
We study certain \emph{separation probabilities} for products of permutations. The archetypal question can be stated as follows: \emph{in the symmetric group $\mathfrak{S}_n$, what is the probability that the elements $1,2,\ldots,k$ are in distinct cycles of the product of two $n$-cycles chosen uniformly randomly?} The answer is surprisingly elegant: the probability is $\frac{1}{k!}$ if $n-k$ is odd and $\frac{1}{k!} + \frac{2}{(k-2)!(n-k+1)(n+k)}$ if $n-k$ is even. This result was originally conjectured by B\'ona \cite{MBRF} for $k=2$ and $n$ odd. Subsequently, Du and Stanley proved it for all $k$ and proposed additional conjectures \cite{DS}. The goal of this paper is to prove these conjectures, and establish generalizations of the above result. Our approach is different from the one used in \cite{DS}.\\

Let us define a larger class of problems. 
Given a tuple $A=(A_1,\ldots,A_k)$ of $k$ disjoint non-empty subsets of $\{1,\ldots,n\}$, we say that a permutation $\pi$ is \emph{$A$-separated} if no cycle of $\pi$ contains elements of more than one of the subsets $A_i$. 
Now, given two integer partitions $\la,\mu$ of $n$, one can wonder about the probability $P_{\la,\mu}(A)$ that the product of two uniformly random permutations of cycle type $\la$ and $\mu$ is $A$-separated. The example presented above corresponds to $A=(\{1\},\ldots,\{k\})$ and $\la=\mu=(n)$. Clearly, the separation probabilities $P_{\la,\mu}(A)$ only depend on $A$ through the size of the subsets $\#A_1,\ldots,\#A_k$, and we shall denote $\psig_{\la,\mu}^{\al}:=P_{\la,\mu}(A)$, where $\alpha=(\#A_1,\ldots,\#A_k)$ is a composition (of size $m\leq n$). Note also that $\psig_{\la,\mu}^{\al}=\psig_{\la,\mu}^{\al'}$ whenever the composition $\al'$ is a permutation of the composition $\al$. Below, we focus on the case $\mu=(n)$ and we further denote $\psig_{\la}^{\al}:=\psig_{\la,(n)}^{\al}$.\\

In this paper, we first express the separation probabilities $\psig_{\la}^{\al}$ as some coefficients in an explicit generating function. Using this expression we then prove the following symmetry property: if $\alpha=(\al_1,\ldots,\al_k)$ and $\beta=(\be_1,\ldots,\be_k)$ are compositions of the same size $m\leq n$ and of the same length $k$, then
\begin{equation} \label{eq:probsep}
\frac{\psig_{\lambda}^\al}{\prod_{i=1}^k \alpha_i!}=\frac{\psig_{\lambda}^{\beta}}{\prod_{i=1}^k \beta_i!}.
\end{equation}
Moreover, for certain partitions $\la$ (including  the cases $\la=(n)$ and $\la=2^N$) we obtain explicit expressions for the  probabilities $\psig_{\la}^{\al}$ for certain partitions $\la$. For instance, the separation probability $\psig_{(n)}^\al$ for the product of two $n$-cycles is found to be
\begin{equation} \label{eq:probsigma}
\psig_{(n)}^\al = \frac{(n-m)!\prod_{i=1}^k \alpha_i!}{(n+k)(n-1)!}\left(\frac{(-1)^{n-m}\binom{n-1}{k-2}}{\binom{n+m}{m-k}}+ \sum_{r=0}^{m-k}\frac{(-1)^r\binom{m-k}{r} \binom{n+r+1}{m}}{\binom{n+k+r}{r}}\right).
\end{equation}
This includes the case $\al=1^k$ proved by Du and Stanley \cite{DS}.\\

Our general expression for the separation probabilities $\psig_{\la}^{\al}$ is derived using a formula obtained in \cite{MV} about \emph{colored} factorizations of the $n$-cycle into two permutations. This formula displays a symmetry which turns out to be of crucial importance for our method. Our approach can in fact be made mostly bijective as explained in Section~\ref{sec:maps}. Indeed, the formula obtained in \cite{MV} builds on a bijection established in \cite{SV}. An alternative bijective proof was given in \cite{BM1} and in Section~\ref{sec:maps} we explain how to concatenate this bijective proof with the constructions of the present paper.


\medskip

\noindent {\bf Outline.} 
In Section~\ref{sec:strategy} we present our strategy for computing the separation probabilities. This involves counting certain colored factorizations of the $n$-cycle. We then gather our main results in Section~\ref{sec:results}. In particular we prove the symmetry property \eqref{eq:probsep} and obtain formulas for the separation probabilities $\psig_{\la}^{\al}$ for certain partitions $\la$ including $\la=(n)$ or when $\la=2^N$. 
In Section~\ref{sec:fixedpoints}, we give formulas relating the separation probabilities $\psig_{\la}^{\al}$ and $\psig_{\la'}^{\al}$ when $\la'$ is a partition obtained from another partition $\la$ by adding some parts of size 1. In Section~\ref{sec:maps}, we indicate how our proofs could be made bijective. We gather a few additional remarks in Section~\ref{sec:conclusion}. \\


\noindent {\bf Notation.} 
We denote $[n]:=\{1,2,\ldots,n\}$. We denote by $\#S$ the cardinality of a set $S$.

A \emph{composition} of an integer $n$ is a tuple $\alpha=(\alpha_1,\alpha_2,\ldots,\alpha_k)$ of positive integer summing to $n$. We then say that $\al$ has \emph{size} $n$ and \emph{length} $\ell(\al)=k$. An \emph{integer partition} is a composition such that the \emph{parts} $\al_i$ are in weakly decreasing order. We use the notation $\lambda \models n$ (resp. $\lambda \vdash n$) to indicate that $\la$ is a composition (resp. integer partition) of $n$. We sometime write integer partitions in multiset notation: writing $\la=1^{n_1},2^{n_2},\ldots,j^{n_j}$ means that $\la$ has $n_i$ parts equal to $i$.

We denote by $\mathfrak{S}_n$ the symmetric group on $[n]$. Given a partition $\lambda$ of $n$, we denote by $\mathcal{C}_{\lambda}$ the set of permutations in $\mathfrak{S}_n$ with cycle type $\lambda$. It is well known that $\#\mathcal{C}_{\lambda}=n!/z_{\lambda}$ where $z_{\lambda}=\prod_i i^{n_i(\lambda)} n_i(\lambda)!$ and $n_i(\lambda)$ is the number of parts equal to $i$ in $\lambda$.

We shall consider symmetric functions in an infinite number of variables $\xx=\{x_1,x_2,\ldots\}$. For any sequence of nonnegative integers, $\al=(\al_1,\al_2,\ldots,\al_k)$ we denote $\xx^{\al}:=x_1^{\al_1}x_2^{\al_2}\ldots x_k^{\al_k}$. We denote by $[\xx^{\al}]f(\xx)$ the coefficient of this monomial in a series $f(\xx)$.
For an integer partition $\la=(\la_1,\ldots,\la_k)$ we denote by $p_{\lambda}(\xx)$ and $m_{\lambda}(\xx)$ respectively the \emph{power symmetric function} and \emph{monomial symmetric function} indexed by $\lambda$ (see e.g. \cite{EC2}). That is, $p_{\lambda}(\xx)=\prod_{i=1}^{\ell(\lambda)} p_{\lambda_i}(\xx)$ where $p_k(\xx)=\sum_{i\geq 1} x_i^k$, and $m_{\lambda}(\xx)=\sum_{\alpha} \xx^\al$ where the sum is over all the distinct sequences $\al$ whose positive parts are $\{\la_1,\la_2,\ldots,\la_k\}$ (in any order). Recall that the power symmetric functions form a basis of the ring of symmetric functions. For a symmetric function $f(\xx)$ we denote by $[p_{\lambda}(\xx)]f(\xx)$ the coefficient of $p_{\lambda}(\xx)$ of the decomposition of $f(\xx)$ in this basis.


\section{Strategy} \label{sec:strategy}
In this section, we first translate the problem of determining the separation probabilities $\psig_\la^\al$ into the problem of enumerating certain sets $\mS_\la^\al$. Then, we introduce a symmetric function $\GG^\al_n(\xx,t)$ whose coefficients in one basis are the cardinalities $\#\mS_\la^\al$, while the coefficients in another basis count certain ``colored'' separated factorizations of the permutation $(1,\ldots,n)$. Lastly, we give exact counting formulas for these colored separated factorizations. Our main results will follow as corollaries in Section~\ref{sec:results}.\\

For a composition $\al=(\al_1,\ldots,\al_k)$ of size $m\leq n$, we denote by $\mA_n^\al$ the set of tuples $A=(A_1,\ldots,A_k)$ of pairwise disjoint subsets of $[n]$ with $\#A_i=\al_i$ for all $i$ in $[k]$. Observe that $\#\mA_n^\al=\binom{n}{\alpha_1,\alpha_2,\ldots,\alpha_k,n-m}$.\\
 
Now, recall from the introduction that $\psig_{\la}^{\al}$ 
is the probability for the product of a uniformly random permutation of cycle type $\la$ with a \emph{uniformly random} $n$-cycle to be $A$-separated for a \emph{fixed} tuple $A$ in $\mA_n^\al$. Alternatively, it can be defined as the probability for the product of a uniformly random permutation of cycle type $\la$ with a \emph{fixed} $n$-cycle to be $A$-separated for a \emph{uniformly random} tuple $A$ in $\mA_n^\al$ (since the only property that matters is that the elements in $A$ are randomly distributed in the $n$-cycle). 
\begin{definition}
For an integer partition $\la$ of $n$, and a composition $\al$ of $m\leq n$, we denote by $\mS_\la^\al$ the set of pairs $(\pi,A)$, where  $\pi$ is a permutation in $\mathcal{C}_\la$ and $A$ is a tuple in $\mA_n^\al$ such that the product $\pi\circ (1,2,...,n)$ is $A$-separated.
\end{definition}\label{def:setS}
From the above discussion we obtain for any composition $\al=(\al_1,\ldots,\al_k)$ of size $m$,
\begin{equation}\label{eq:probtoenumsig}
\psig_{\lambda}^\al = \frac{\#\mS_\la^\al}{\binom{n}{\alpha_1,\alpha_2,\ldots,\alpha_k,n-m}\#\mathcal{C}_{\lambda}}. 
\end{equation}\\ 

Enumerating the sets $\mS_\la^\al$ directly seems rather challenging. However, we will show below how to enumerate a related class of ``colored'' separated permutations denoted by $\mT^\al_{\ga}(r)$. We define a \emph{cycle coloring} of a permutation $\pi\in\mathfrak{S}_n$ in $[q]$ to be a mapping $c$ from $[n]$ to $[q]$ such that if $i,j\in[n]$ belong to the same cycle of $\pi$ then $c(i)=c(j)$. We think of $[q]$ as the set of \emph{colors}, and $c^{-1}(i)$ as set of \emph{elements colored $i$}. 
\begin{definition}\label{def:setT}
Let $\ga=(\ga_1,\ldots,\ga_\ell)$ be a composition of size $n$ and length $\ell$, and let $\alpha=(\al_1,\ldots,\al_k)$ be a composition of size $m\leq n$ and length $k$. For a nonnegative integer $r$ we define $\mT^\al_{\ga}(r)$ as the set of quadruples $(\pi,A,c_1,c_2)$, where $\pi$ is a permutation of $[n]$, $A=(A_1,\ldots,A_k)$ is in $\mA^\al_n$, and 
\begin{compactitem}
\item[(i)] $c_1$ is a cycle coloring of $\pi$ in $[\ell]$ such that there are $\ga_i$ element colored $i$ for all $i$ in $[\ell]$,
\item[(ii)] $c_2$ is a cycle coloring of the product $\pi\circ(1,2,\ldots,n)$ in $[k+r]$ such that every color in $[k+r]$ is used and for all $i$ in $[k]$ the elements in the subset $A_i$ are colored $i$.
\end{compactitem}
\end{definition}
Note that condition (ii) in Definition~\ref{def:setT} and the definition of cycle coloring implies that the product $\pi\circ(1,2,\ldots,n)$ is $A$-separated.\\

In order to relate the cardinalities of the sets $\mS^\al_{\lambda}$ and $\mT^\al_{\ga}(r)$, it is convenient to use symmetric functions (in the variables $\xx=\{x_1,x_2,x_3,\ldots\}$). 
Namely, given a composition $\al$ of $m\leq n$, we define
\[
\GG^\al_n(\xx,t) := \sum_{\lambda \vdash n} p_{\lambda}(\xx) \sum_{ (\pi,A) \in \mS_\la^\al} t^{\uc(\pi,A)},
\]
where the outer sum runs over all the integer partitions of $n$, and $\uc(\pi,A)$ is the number of cycles of the product $\pi\circ(1,2,\ldots,n)$ containing none of the elements in $A$. Recall that the power symmetric functions $p_{\lambda}(\xx)$ form a basis of the ring of symmetric functions, so that the contribution of a partition $\la$ to $\GG^\al_n(\xx,t)$ can be recovered by extracting the coefficient of $p_{\lambda}(\xx)$ in this basis:
\begin{equation} \label{genseriestosig}
\nsig^\al_{\lambda} = [p_{\lambda}(\xx)] \,\, \GG^\al_n(\xx,1).
\end{equation}
As we prove now, the sets $\mT^\al_{\ga}(r)$ are related to the coefficients of $\GG^\al_n(\xx,t)$ in the basis of monomial symmetric functions.
\begin{proposition}\label{prop:generating-function}
If $\al$ is a composition of length $k$, then
\begin{equation} \label{eq:colsepmap}
\GG_n^\al(\xx,t+k)=\sum_{\ga \vdash n} m_{\ga}(\xx)  \sum_{r\geq 0}\binom{t}{r} \,\#\mT^\al_{\gamma}(r),
\end{equation}
where the outer sum is over all integer partitions of $n$, and $\displaystyle \binom{t}{r}:=\frac{t(t-1)\cdots (t-r+1)}{r!}$.
\end{proposition}

\begin{proof} Since both sides of \eqref{eq:colsepmap} are polynomial in $t$ and symmetric function in $\xx$ it suffices to show that for any nonnegative integer $t$ and any partition $\ga$ the coefficient of $\xx^\ga$ is the same on both sides of \eqref{eq:colsepmap}.
We first determine the coefficient $\displaystyle [\xx^{\ga}]\GG_n^\al(\xx,t+k)$ when $t$ is a nonnegative integer.
Let $\la$ be a partition, and $\pi$ be a permutation of cycle type $\la$. Then the symmetric function $p_{\la}(\xx)$ can be interpreted as the generating function of the cycle colorings of $\pi$, that is, for any sequence $\ga=(\ga_1,\ldots,\ga_{\ell})$ of nonnegative integers, the coefficient $[\xx^{\ga}]p_{\la}(\xx)$ is the number of cycle colorings of $\pi$ such that $\ga_i$ elements are colored $i$, for all $i> 0$. Moreover, if $\pi$ is $A$-separated for a certain tuple $A=(A_1,\ldots,A_k)$ in $\mA^\al_n$, then $(t+k)^{\uc(S,\pi)}$ represents the number of cycle colorings of the permutation $\pi\circ(1,2,\ldots,n)$ in $[k+t]$ (not necessarily using every color) such that for all $i\in[k]$ the elements in the subset $A_i$ are colored $i$. Therefore, for a partition $\ga$ and a nonnegative integer $t$, the coefficient $[\xx^{\ga}]\GG_n^\al(\xx,t+k)$ counts the number of quadruples $(\pi,A,c_1,c_2)$, where $\pi,A,c_1,c_2$ are as in the definition of $\mT^\al_{\ga}(t)$ except that $c_2$ might actually use only a subset of the colors $[k+t]$. Note however that all the colors in $[k]$ will necessarily be used by $c_2$, and that we can partition the quadruples according to the subset of colors used by $c_2$. This gives 
$$[\xx^{\ga}]\GG_n^\al(\xx,t+k)=\sum_{ r\geq 0}\binom{t}{r} \,\#\mT^\al_{\gamma}(r).$$ 
Now extracting the coefficient of $\xx^{\ga}$ in the right-hand side of~\eqref{eq:colsepmap} gives the same result. This completes the proof.
\end{proof}

In order to obtain an explicit expression for the series $\GG^\al_n(\xx,t)$ it remains to enumerate the sets $\mT^\al_\ga(r)$ which is done below.

\begin{proposition}\label{prop:cardT}
Let $r$ be a nonnegative integer, let $\al$ be a composition of size $m$ and length $k$, and let $\ga$ be a partition of size $n\geq m$ and length $\ell$. Then the set $\mT^\al_\ga(r)$ specified by Definition~\ref{def:setT} has cardinality
\begin{equation}\label{eq:cardT}
\#\mT^\al_\ga(r)=\frac{n(n-\ell)!(n-k-r)!}{(n-k-\ell-r+1)!}\, \binom{n+k-1}{n-m-r},
\end{equation}
if $n-k-\ell-r+1\geq 0$, and 0 otherwise.
\end{proposition}

The rest of this section is devoted to the proof of Proposition~\eqref{prop:cardT}. 
In order to count the quadruples $(\pi,A,c_1,c_2)$ satisfying Definition~\ref{def:setT}, we shall start by choosing $\pi,c_1,c_2$ before choosing the tuple $A$. For compositions $\ga=(\ga_1,\ldots,\ga_\ell)$, $\de=(\de_1,\ldots,\de_{\ell'})$ of $n$ we denote by $\mB_{\ga,\de}$ the set of triples $(\pi,c_1,c_2)$, where $\pi$ is a permutation of $[n]$, $c_1$ is a cycle coloring of $\pi$ such that $\ga_i$ elements are colored $i$ for all $i\in[\ell]$, and $c_2$ is a cycle coloring of the permutation $\pi\circ (1,2,\ldots,n)$ such that $\de_i$ elements are colored $i$ for all $i\in[\ell']$. The problem of counting such sets was first considered by Jackson \cite{J} who actually enumerated the union $\mathcal{B}^n_{i,j}:=\displaystyle \bigcup_{\ga,\de\models n,~\ell(\ga)=i,~\ell(\de)=j}\mB_{\ga,\de}$ using representation theory. It was later proved in \cite{MV} that 
\begin{equation}\label{eq:colored-factorizations}
\#\mB_{\ga,\de}=\frac{n(n-\ell)!(n-\ell')!}{(n-\ell-\ell'+1)!}, 
\end{equation}
if $n-\ell-\ell'+1\geq 0$, and 0 otherwise. 
The proof of~\eqref{eq:colored-factorizations} in \cite{MV} uses a refinement of a bijection designed in \cite{SV} in order to prove Jackson's formula for $\#\mathcal{B}^n_{i,j}$. 
 Another bijective proof of~\eqref{eq:colored-factorizations} is given in \cite{BM1}, and we shall discuss it further in Section~\ref{sec:maps} (a proof of~\eqref{eq:colored-factorizations} using representation theory can be found in \cite{EV}).\\ 

One of the striking features of the counting formula~\eqref{eq:colored-factorizations} is that it depends on the compositions $\ga$, $\de$ only through their lengths $\ell$, $\ell'$. This ``symmetry'' will prove particularly handy for enumerating $\mT^\al_\ga(r)$. Let $r$, $\al$, $\ga$ be as in Proposition~\ref{prop:cardT}, and let $\de=(\de_1,\ldots,\de_{k+r})$ be a composition of $n$ of length $k+r$. 
We denote by $\mT^\al_{\ga,\de}$ the set of quadruples $(\pi,A,c_1,c_2)$ in $\mT^\al_\ga(r)$ such that the cycle coloring $c_2$ has $\de_i$ elements colored $i$ for all $i$ in $[k+r]$ (equivalently, $(\pi,c_1,c_2)\in\mB_{\ga,\de}$). We also denote $\dd^\al_\de:=\prod_{i=1}^k{\de_i \choose \al_i}$. It is easily seen that for any triple $(\pi,c_1,c_2)\in\mB_{\ga,\de}$, the number $\dd^\al_\de$ counts the tuples $A\in\mA_n^\al$ such that $(\pi,A,c_1,c_2)\in\mT^\al_{\ga,\de}$. Therefore,
$$ \#\mT^\al_\ga(r)=\sum_{\de\models n,~\ell(\de)=k+r}\#\mT^\al_{\ga,\de}=\sum_{\de\models n,~\ell(\de)=k+r}\dd^\al_\de\, \#\mB_{\ga,\de},$$
where the sum is over all the compositions of $n$ of length $k+r$. Using~\eqref{eq:colored-factorizations} then gives
$$ \#\mT^\al_\ga(r)=\frac{n(n-\ell)!(n-k-r)!}{(n-k-\ell-r+1)!}\sum_{\de\models n,~\ell(\de)=k+r}\dd^\al_\de$$
if $n-k-\ell-r+1\geq 0$, and 0 otherwise. In order to complete the proof of Proposition~\ref{prop:cardT}, it only remains to prove the following lemma.

\begin{lemma} \label{lem:nb-markings}
If $\al$ has size $m$ and length $k$, then
$$\sum_{\de\models n,~\ell(\de)=k+r}\dd^\al_\de=\binom{n+k-1}{n-m-r}.$$
\end{lemma}

\begin{proof}
We give a bijective proof illustrated in Figure~\ref{fig:bijmarking}. One can represent a composition $\de=(\de_1,\ldots,\de_{k+r})$ as a sequence of rows of boxes (the $i$th row has $\de_i$ boxes). With this representation, $\dd^\al_\de:=\prod_{i=1}^k{\de_i \choose \al_i}$ is the number of ways of choosing $\al_i$ boxes in the $i$th row of $\delta$ for $i=1,\ldots,k$. Hence $\sum_{\de\models n,~\ell(\de)=k+r}\dd^\al_\de$ counts \emph{$\al$-marked compositions}  of size $n$ and length $k+r$, that is, sequences of $k+r$ non-empty rows of boxes with some marked boxes in the first $k$ rows, with a total of $n$ boxes, and $\al_i$ marks in the $i$th row for $i=1,\ldots,k$; see Figure~\ref{fig:bijmarking}. 
Now $\al$-marked compositions of size $n$ and length $k+r$ are clearly in bijection (by adding a marked box to each of the rows $1,\ldots,k$, and marking the last box of each of the rows $k+1,\ldots,k+r$) with $\al'$-marked compositions of size $n+k$ and length $k+r$ \emph{such that the last box of each row is marked}, where $\al'=(\al_1+1,\al_2+1,\ldots,\al_k+1,1,1,\ldots,1)$ is a composition of length $k+r$. Lastly, these objects are clearly in bijection (by concatenating all the rows) with sequences of $n+k$ boxes with $m+k+r$ marks, one of which is on the last box. There are $\binom{n+k-1}{n-m-r}$ such sequences, which concludes the proof of Lemma~\ref{lem:nb-markings} and Proposition~\ref{prop:cardT}.
\end{proof}

\fig{width=\linewidth}{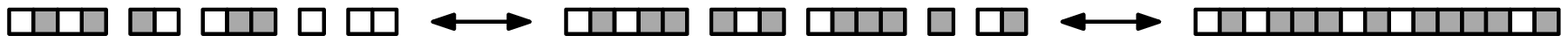}{A $(2,1,2)$-marked composition of size $n=12$ and length $5$ and its bijective transformation into a sequence $n+k=15$ boxes with $m+k+r=5+3+2=10$ marks, one of which is on the last box.}


\section{Main results} \label{sec:results}
In this section, we exploit Propositions~\ref{prop:generating-function} and~\ref{prop:cardT} in order to derive our main results. All the results in this section will be consequences of the following theorem.

\begin{theorem}\label{thm:GF}
For any composition $\al$ of $m\leq n$ of length $k$, the generating function $\GG_n^\al(\xx,t+k)$ in the variables $t$ and $\xx=\{x_1,x_2,\ldots\}$ has the following explicit expression in the bases $m_{\la}(\xx)$ and $\binom{t}{r}$:
\begin{equation}\label{eq:GF-explicit}
\GG_n^\al(\xx,t+k) = \sum_{r=0}^{n-m}\binom{t}{r}\binom{n+k-1}{n-m-r}\sum_{\la \vdash n,~ \ell(\la)\leq n-k-r+1}  \frac{n(n-\ell(\la))!(n-k-r)!}{(n-k-r-\ell(\la)+1)!} ~ m_{\la}(\xx).
\end{equation}
Moreover, for any partition $\la$ of $n$, one has $\displaystyle \nsig_\la^\al=[p_\la(\xx)]\GG_n^\al(\xx,1)$ and $\displaystyle \psig_{\la}^\al=\frac{\nsig_\la^\al}{{n \choose \al_1,\al_2,\ldots,\al_k,n-m}\#\mathcal{C}_\la}$. 
\end{theorem}
Theorem~\ref{thm:GF} is the direct consequence of Propositions~\ref{prop:generating-function} and~\ref{prop:cardT}. One of the striking features of~\eqref{eq:GF-explicit} is that the expression of $\GG_n^\al(\xx,t+k)$ depends on $\al$ only through its size and length.
This ``symmetry property'' then obviously also holds for $\nsig_{\la}^\al=[p_{\la}(\xx)]\GG_n^\al(\xx,1)$, and translates into the formula \eqref{eq:probsep} for separation probabilities as stated below. 

\begin{corollary}\label{cor:sepsym}
Let $\la$ be a partition of $n$, and let $\al=(\al_1,\ldots,\al_k)$ and $\be=(\be_1,\ldots,\be_k)$ be compositions of the same size $m$ and length $k$. Then, 
\begin{equation} \label{eq:sepsym}
\nsig_{\la}^\al=\nsig_{\la}^{\beta},
\end{equation} 
or equivalently, in terms of separation probabilities, $\displaystyle ~\frac{\psig_{\la}^\al}{\prod_{i=1}^k \alpha_i!}=\frac{\psig_{\la}^{\beta}}{\prod_{i=1}^k \beta_i!}$.
\end{corollary}

We now derive explicit formulas for the separation probabilities for the product of a uniformly random permutation $\pi$, with particular constraints on its cycle type, with a uniformly random $n$-cycle. We focus on two constraints: the case where $\pi$ is required to have $p$ cycles, and the case where $\pi$ is a fixed-point-free involution (for $n$ even).

\subsection{Case when $\pi$ has exactly $p$ cycles}
Let $\mC(n,p)$ denote the set of permutations of $[n]$ having $p$ cycles. Recall that the numbers $c(n,p)=\#\mC(n,p)=[x^p]x(x+1)(x+2)\cdots(x+n-1)$ are called the \emph{signless Stirling numbers of the first kind}. We denote by $\psig^{\al}(n,p)$ the probability that the product of a uniformly random permutation in $\mC(n,p)$ with a uniformly random $n$-cycle is $A$-separated for a given set $A$ in $\mA_n^\al$. By a reasoning similar to the one used in the proof of \eqref{eq:probtoenumsig}, one gets
\begin{equation}\label{eq:probtoenumsigP}
\psig^{\al}(n,p)=\frac{1}{{n \choose \al_1,\al_2,\ldots,\al_k,n-m}c(n,p)}\sum_{\la \vdash n, \ell(\la)=p} \nsig_{\la}^\al.
\end{equation} 
We now compute the probabilities $\psig^{\al}(n,p)$ explicitly.

\begin{theorem} \label{thm:casePcycles}
Let $\alpha$ be a composition of $m$ with $k$ parts. Then,
\begin{equation} \label{eq:probsigmaP}
\psig^{\al}(n,p) = \frac{(n-m)! \prod_{i=1}^k \alpha_i!}{c(n,p)}\sum_{r=0}^{n-m} \binom{1-k}{r} \binom{n+k-1}{n-m-r}\frac{c(n-k-r+1,p)}{(n-k-r+1)!},
\end{equation}
where $c(n,p)$ are signless Stirling numbers of the first kind.
\end{theorem}

For instance, Theorem \ref{thm:casePcycles}  in the case $m=n$ gives the probability that the cycles of the product of a uniformly random permutation in $\mC(n,p)$ with a uniformly random $n$-cycle refine a given set partition of $[n]$ having blocks of sizes $\al_1,\al_2,\ldots,\al_k$. This probability is found to be 
$$
\psig^{\al}(n,p)= \frac{\prod_{i=1}^k \alpha_i!}{c(n,p)}\,\frac{c(n-k+1,p)}{(n-k+1)!}.
$$

We now prove Theorem~\ref{thm:casePcycles}. Via~\eqref{eq:probtoenumsigP}, this amounts to enumerating $\mS^{\al}(n,p):=\bigcup_{\la \vdash n, \ell(\la)=p}\mS_{\la}^\al$, and using Theorem~\ref{thm:GF} one gets
\begin{eqnarray}\label{eq:formsigmaPeq1}
\nsig^\al(n,p)&=&\sum_{\la \vdash n,\ell(\la)=p} [p_{\la}(\xx)] \, \GG_n^\al(\xx,1)\nonumber \\
&=& \sum_{r=0}^{n-m}\binom{1-k}{r}\binom{n+k-1}{n-m-r}\sum_{\ell=1}^{n-k-r+1}   \frac{n(n-\ell)!(n-k-r)!}{(n-k-r-\ell+1)!}\, A(n,p,\ell),
\end{eqnarray}
where $\displaystyle A(n,p,\ell):=\sum_{\mu \vdash n,~ \ell(\mu)=p} [p_{\mu}(\xx)] \sum_{\la \vdash n,~ \ell(\la)=\ell} m_{\la}(\xx)$.
The next lemma gives a formula for $A(n,p,\ell)$.
\begin{lemma}\label{lem:coeff-k-cycles}
For any positive integers $p,\ell\leq n$
\begin{equation}\label{eq:coeff-k-cycles}
\sum_{\mu \vdash n,~ \ell(\mu)=p} [p_{\mu}(\xx)] \sum_{\la \vdash n,~ \ell(\la)=\ell} m_{\la}(\xx) ~=~ {n-1 \choose \ell-1}\frac{(-1)^{\ell-p} c(\ell,p)}{\ell!},
\end{equation}
where $c(a,b)$ are the signless Stirling numbers of the first kind.
\end{lemma}

\begin{proof}
For this proof we use the principal specialization of symmetric functions, that is, their evaluation at $\xx=1^a:=\{1,1,\ldots,1,0,0\ldots\}$ ($a$ ones). Since $p_{\ga}(1^a)=a^{\ell(\ga)}$ for any positive integer $a$, one gets

$$\sum_{\la \vdash n,~ \ell(\la)=\ell} m_{\la}(1^a)=\sum_{p=1}^n a^p \sum_{\mu \vdash n,~ \ell(\mu)=p} [p_{\mu}(\xx)] \sum_{\la \vdash n,~ \ell(\la)=\ell} m_{\la}(\xx).
$$
The right-hand side of the previous equation is a polynomial in $a$, and by extracting the coefficient of $a^p$  one gets
\begin{equation}\label{eq:extract-coeff}
\sum_{\mu \vdash n,~ \ell(\mu)=p} [p_{\mu}(\xx)] \sum_{\la \vdash n,~ \ell(\la)=\ell} m_{\la}(\xx) =[a^p]\sum_{\la \vdash n,~ \ell(\la)=\ell} m_{\la}(1^a).\nonumber
\end{equation}
Now, for any partition $\la$,  $m_{\la}(1^a)$ counts the $a$-tuples of nonnegative integers such that the positive ones are the same as the parts of $\la$ (in some order). Hence $\displaystyle \sum_{\la \vdash n,~ \ell(\la)=\ell} m_{\la}(1^a)$ counts the $a$-tuples of nonnegative integers with $\ell$ positive ones summing to $n$. This gives, 
\begin{equation}\label{eq:extract-coeff2}
\sum_{\la \vdash n,~ \ell(\la)=\ell} m_{\la}(1^a)={n-1 \choose \ell-1}{a \choose \ell}.\nonumber
\end{equation}
Extracting the coefficient of $a^p$ gives~\eqref{eq:coeff-k-cycles} since $\displaystyle [a^p]{a \choose \ell}=\frac{(-1)^{\ell-p}\,c(\ell,p)}{\ell!}$.
\end{proof}

Using Lemma~\ref{lem:coeff-k-cycles} in~\eqref{eq:formsigmaPeq1} gives
\begin{equation} \label{eq:formsigmaPeq2}
\nsig^\al(n,p) = n! \sum_{r\geq 0}^{n-m} \binom{1-k}{r} \binom{n+k-1}{n-m-r} \sum_{\ell=1}^{n-k-r+1} \binom{n-k-r}{\ell-1}\frac{(-1)^{\ell-p}c(\ell,p)}{\ell!},
\end{equation} 
which we simplify using the following lemma.
\begin{lemma}\label{lem:simplifyP}
For any nonnegative integer $a$, $\displaystyle \sum_{q=0}^{a} \binom{a}{q}\frac{(-1)^{q+1-p}\,c(q+1,p)}{(q+1)!} = \frac{c(a+1,p)}{(a+1)!}$.
\end{lemma}
\begin{proof}
The left-hand side equals $[x^p] \sum_{q=0}^a \binom{a}{q}\binom{x}{q+1}$. Using the Chu-Vandermonde identity this equals $[x^p] \binom{x+a}{a+1}$ which is precisely the right-hand side.
\end{proof}
Using Lemma~\ref{lem:simplifyP} in~\eqref{eq:formsigmaPeq2} gives
\begin{equation} \label{eq:formsigmaPeq3}
\nsig^\al(n,p)=n! \sum_{r=0}^{n-m} \binom{1-k}{r} \binom{n+k-1}{n-m-r}\frac{c(n-k-r+1,p)}{(n-k-r+1)!} ,
\end{equation}
which is equivalent to~\eqref{eq:probsigmaP} via~\eqref{eq:probtoenumsig}. This completes the proof of Theorem~\ref{thm:casePcycles}. \hfill $\square$\\

In the case $p=1$, the expression \eqref{eq:probsigmaP} for the probability $\psig^{\al}(1)=\psig^{\al}_{(n)}$ can be written as a sum of $m-k$ terms instead. We state this below.
\begin{corollary}\label{cor:alternative-sum}
Let $\alpha$ be a composition of $m$ with $k$ parts. Then the separation probabilities $\psig_{(n)}^\al$ (separation for the product of two uniformly random $n$-cycles) are 
\begin{equation} \nonumber
\psig_{(n)}^\al= \frac{(n-m)!\prod_{i=1}^k \alpha_i!}{(n+k)(n-1)!}\left(\frac{(-1)^{n-m}\binom{n-1}{k-2}}{\binom{n+m}{m-k}}+ \sum_{r=0}^{m-k}\frac{(-1)^r\binom{m-k}{r} \binom{n+r+1}{m}}{\binom{n+k+r}{r}}\right).
\end{equation}
\end{corollary}
The equation in Corollary~\ref{cor:alternative-sum}, already stated in the introduction, is particularly simple when $m-k$ is small. For $\al=1^k$ (i.e. $m=k$) one gets the result stated at the beginning of this paper:
\begin{equation} 
\label{probkones}
\psig^{1^k}_{(n)}=\begin{cases}
\frac{1}{k!} &\text{ if } n-k \text{ odd,}\\
\frac{1}{k!} + \frac{2}{(k-2)!(n-k+1)(n+k)} &\text{ if } n-k \text{ even.}
\end{cases}
\end{equation}\\

In order to prove Corollary~\ref{cor:alternative-sum} we start with the expression obtained by setting $p=1$ in \eqref{eq:probsigmaP}:
\begin{eqnarray}\label{eq:formsigmaPeq4}
\psig^\al_{(n)}&=&\frac{(n-m)!\prod_{i=1}^k \alpha_i!}{(n-1)!}\sum_{r=0}^{n-m} \binom{1-k}{r} \frac{1}{n-k-r+1} \binom{n+k-1}{n-m-r}\nonumber \\
&=&\frac{(n-m)!\prod_{i=1}^k \alpha_i!}{(n-1)!}[x^{n-m}](1+x)^{1-k}\sum_{r=0}^{n+k-1}\frac{x^r}{r+m-k+1}{n+k-1 \choose r}.
\end{eqnarray}
We now use the following polynomial identity.
\begin{lemma}\label{lem:change-summation}
For nonnegative integers $a,b$, one has the following identity between polynomials in~$x$:
\begin{equation}\label{eq:change-summation}
\sum_{i=0}^a\frac{x^i}{i+b+1}{a \choose i}=\frac{1}{(a+1)} 
\left( \frac{1}{{a+b+1 \choose b}(-x)^{b+1}}- \sum_{i=0}^b\frac{{b \choose i}(x+1)^{a+i+1}}{{a+i+1 \choose i}(-x)^{i+1}}\right).
\end{equation}
\end{lemma}

\begin{proof}
It is easy to see that the left-hand side of~\eqref{eq:change-summation} is equal to $\frac{1}{x^{b+1}}\int_{0}^x (1+t)^at^bdt$. Now this integral can be computed via integration by parts. By a simple induction on $b$, this gives the right-hand side of~\eqref{eq:change-summation}.
\end{proof}

Now using~\eqref{eq:change-summation} in \eqref{eq:formsigmaPeq4}, with $a=n+k-1$ and $b=m-k$, gives
\begin{eqnarray}
\psig^\al_{(n)} &=& \frac{(n-m)!\prod_{i=1}^k \alpha_i!}{(n+k)(n-1)!} [x^{n-m}] \left( \frac{(1+x)^{1-k}}{\binom{n+m}{m-k}(-x)^{m-k+1} }-\sum_{r=0}^{m-k} \frac{\binom{m-k}{r} (1+x)^{n+r+1}}{\binom{n+k+r}{r} (-x)^{r+1}}\right)\nonumber\\
&= &\frac{(n-m)!\prod_{i=1}^k \alpha_i!}{(n+k)(n-1)!}\left(\frac{(-1)^{n-m}\binom{n-1}{k-2}}{\binom{n+m}{m-k}}+ \sum_{r=0}^{m-k}\frac{(-1)^r\binom{m-k}{r} \binom{n+r+1}{m}}{\binom{n+k+r}{r}}\right).\nonumber
\end{eqnarray}

This completes the proof of Corollary~\ref{cor:alternative-sum}. \hfill $\square$

\subsection{Case when $\pi$ is a fixed-point-free involution}\label{subsec:involution}
Given a composition $\alpha$ of $m\leq 2N$ with $k$ parts, we define 
\[
H_{N}^{\alpha}(t):= \sum_{(\pi,A) \in \mS_{2^N}^{\alpha}} t^{\uc(\pi,A)},
\]
where $\uc(\pi,A)$ is the number of cycles of the product $\pi \circ (1,2,\ldots,2N)$ containing none of the elements of $A$ and where $\pi$ is a fixed-point-free involution of $[2N]$. Note that $H_{N}^{\alpha}(t) = [p_{2^N}(\xx)]\, G^{\alpha}_{2N}(\xx,t)$. We now give an explicit expression for this series.

\begin{theorem} \label{thm:GF-Maps}
For any composition $\alpha$ of $m\leq 2N$ of length $k$, the generating series $H_{N}^{\alpha}(t+k)$ is given by
\begin{equation}\label{eq:GF-Mapsexplicit}
H_{N}^{\alpha}(t+k) = N\sum_{r=0}^{\min(2N-m,N-k+1)} \binom{t}{r}\binom{2N+k-1}{2N-m-r} 2^{k+r-N} \frac{(2N-k-r)!}{(N-k-r+1)!}.
\end{equation}
Consequently the separation probabilities for the product of a fixed-point-free involution with a $2N$-cycle are given by
\begin{equation}\label{eq:separation-involution}
\psig_{2^N}^\al=\frac{\prod_{i=1}^k\al_i!}{(2N-1)!(2N-1)!!}\sum_{r=0}^{\min(2N-m,N-k+1)} \binom{1-k}{r}\binom{2N+k-1}{2N-m-r} 2^{k+r-N-1} \frac{(2N-k-r)!}{(N-k-r+1)!}.
\end{equation}
\end{theorem}

\begin{rmk}
It is possible to prove Theorem~\ref{thm:GF-Maps} directly using ideas similar to the ones used to prove Theorem~\ref{thm:GF} in Section~\ref{sec:strategy}. This will be explained in more detail in Section~\ref{sec:maps}. In the proof given below, we instead obtain Theorem~\ref{thm:GF-Maps} as a consequence of Theorem~\ref{thm:GF}.
\end{rmk}

The rest of this section is devoted to the proof of Theorem~\ref{thm:GF-Maps}. Since $H_{N}^{\alpha}(t) = [p_{2^N}(\xx)]\, G^{\alpha}_{2N}(\xx,t)$, Theorem~\ref{thm:GF} gives
\begin{eqnarray} \label{eq:fpfinv}
\lefteqn{H_{N}^{\alpha}(t+k)= }\\
& &
\sum_{r=0}^{2N-m} \binom{t}{r} \binom{2N+k-1}{2N-m-r} \sum_{s=0}^{N-k-r+1} \frac{2N(N-s)!(2N-k-r)!}{(N-k-r-s+1)!} [p_{2^N}(\xx)] \sum_{\la \vdash 2N,~ \ell(\la)=N+s}m_{\la}(\xx). \nonumber
\end{eqnarray}
We then use the following result.
\begin{lemma} \label{lemMaps}
For any nonnegative integer $s\leq N$,
\[
[p_{2^N}(\xx)] \sum_{\la \vdash 2N,~\ell(\la)=N+s} m_{\la}(\xx)~=~ \frac{(-1)^s}{2^s s!(N-s)!}.
\]
\end{lemma}
\begin{proof}
For partitions $\la,\mu$ of $n$, we denote $S_{\la,\mu}=[p_\la(\xx)]m_\mu(\xx)$ and $R_{\la,\mu}=[m_\la(\xx)]p_\mu(\xx)$. The matrices $S=(S_{\la,\mu})_{\la,\mu \vdash n}$ and $R=(R_{\la,\mu})_{\la,\mu \vdash n}$ are the transition matrices between the bases $\{p_{\la}\}_{\la, \vdash n}$ and $\{m_{\la}\}_{\la \vdash n}$ of symmetric functions of degree $n$, hence $S=R^{-1}$. Moreover the matrix $R$ is easily seen to be lower triangular in the \emph{dominance order} of partitions, that is, $R_{\la,\mu}=0$ unless $\la_1+\la_2+\cdots +\la_i \leq \mu_1+\mu_2+\cdots+ \mu_i$ for all $i\geq 1$ (\cite[Prop. 7.5.3]{EC2}). Thus the matrix $S=R^{-1}$ is also lower triangular in the dominance order. 
Since the only partition of $2N$ of length $N+s$ that is not larger than the partition $2^N$ in the dominance order is $1^{2s}2^{N-s}$, one gets
\begin{equation}\label{eq:extract1}
[p_{2^N}(\xx)] \sum_{\la \vdash 2N,~\ell(\la)=N+s} m_{\la}(\xx) = [p_{2^N}(\xx)] \,m_{1^{2s}2^{N-s}}(\xx). 
\end{equation}
To compute this coefficient we use the standard scalar product $\langle\cdot,\cdot\rangle$ on symmetric functions (see e.g. \cite[Sec. 7]{EC2}) defined by $\langle p_{\la},p_{\mu}\rangle = z_{\la}$ if $\la=\mu$ and 0 otherwise, where $z_{\la}$ was defined at the end of Section~\ref{sec:intro}. 
From this definition one immediately gets 
\begin{equation}\label{eq:extract2}
[p_{2^N}] \,m_{1^{2s}2^{N-s}} = \frac{1}{z_{2^N}}\langle p_{2^N}, m_{1^{2s}2^{N-s}}\rangle= \frac{1}{N!2^N}\langle p_{2^N}, m_{1^{2s}2^{N-s}}\rangle.
\end{equation}
Let $\{h_{\la}\}$ denote the basis of the complete symmetric functions. It is well known that $\langle h_{\la},m_{\mu}\rangle=1$ if $\la=\mu$ and 0 otherwise, therefore $\langle p_{2^N}, m_{1^{2s}2^{N-s}}\rangle=[h_{1^{2s}2^{N-s}}]p_{2^N}$. Lastly, since $p_{2^N}=(p_2)^N$ and $p_2= 2h_2-h_1^2$ one gets 
\begin{equation}\label{eq:extract3}
\langle p_{2^N}, m_{1^{2s}2^{N-s}}\rangle=[h_{1^{2s}2^{N-s}}]p_{2^N}= [h_1^{2s}h_2^{N-s}]\, (2h_2-h_1^2)^N=2^{N-s}(-1)^s {N\choose s}.
\end{equation}
Putting together \eqref{eq:extract1}, \eqref{eq:extract2} and \eqref{eq:extract3} completes the proof.
\end{proof}

By Lemma~\ref{lemMaps}, Equation \eqref{eq:fpfinv} becomes
\begin{align*}
H_{N}^{\alpha}(t+k)&= \sum_{r=0}^{2N-m} \binom{t}{r} \binom{2N+k-1}{2N-m-r} \sum_{s=0}^{N-k-r+1} \frac{2N(N-s)!(2N-k-r)!}{(N-k-r-s+1)!} \frac{(-1)^s}{2^s s!(N-s)!}\\
&=2N\sum_{r=0}^{2N-m} \binom{t}{r} \binom{2N+k-1}{2N-m-r}\frac{(2N-k-r)!}{(N-k-r+1)!} \sum_{s=0}^{N-k-r+1} \binom{N-k-r+1}{s}\frac{(-1)^s}{2^s}\\
&=2N\sum_{r=0}^{\min(2N-m,N-k+1)} \binom{t}{r}\binom{2N+k-1}{2N-m-r} \frac{(2N-k-r)!}{(N-k-r+1)!} \frac{1}{2^{N-k-r+1}},
\end{align*}
where the last equality uses the binomial theorem. This completes the proof of Equation \eqref{eq:GF-Mapsexplicit}. Equation \eqref{eq:separation-involution} then immediately follows from the case $t=1-k$ of \eqref{eq:GF-Mapsexplicit} via \eqref{eq:probtoenumsig}.
 This completes the proof of Theorem~\ref{thm:GF-Maps}. \hfill$\square$


\section{Adding fixed points to the permutation $\pi$}\label{sec:fixedpoints}
In this section we obtain a relation between the separation probabilities $\psig_{\la}^\al$ and $\psig_{\la'}^\al$, when the partition $\la'$ is obtained from $\la$ by adding some parts of size 1. Our main result is given below.

\begin{theorem}\label{thm:addfixpt} 
Let $\la$ be a partition of $n$ with parts of size at least $2$ and let $\la'$ be the partition obtained from $\la$ by adding $r$ parts of size 1. Then for any composition $\al=(\al_1,\ldots,\al_k)$ of $m\leq n+r$ of length $k$,
\begin{equation} \label{eq:addfixpoints}
\#\mS^{\alpha}_{\la'} = \sum_{p=0}^{m-k} \left(\frac{n+p}{n} \binom{n+m+r-p}{n+m} + \frac{m-p}{n} \binom{n+m+r-p-1}{n+m} \right)\binom{m-k}{p} \#\mS^{(m-k-p+1,1^{k-1})}_{\la}.
\end{equation}
Equivalently, in terms of separation probabilities,
\begin{equation}
\label{eq:addfixpointsProb}
\sigma^{\alpha}_{\la'} = \frac{n!}{\binom{n+r}{\alpha_1,\ldots,\al_k,n+r-m}\binom{n+r}{r}}\sum_{p=0}^{m-k}  \frac{\left(\frac{n+p}{n} \binom{n+m+r-p}{n+m} + \frac{m-p}{n} \binom{n+m+r-p-1}{n+m} \right) \binom{m-k}{p}}{(n-m+p)!(m-k-p+1)!}\, \sigma^{(m-k-p+1,1^{k-1})}_{\la}.
\end{equation}
\end{theorem}
For instance, when $\al=1^k$ Theorem~\ref{thm:addfixpt} gives 
\begin{equation} \nonumber 
\sigma^{1^k}_{\la'} = \frac{\binom{n+r-k}{r}}{\binom{n+r}{r}^2}\left( \binom{n+r+k}{n+k} + \frac{k}{n}\binom{n+r+k-1}{n+k} \right)\sigma^{1^k}_{\la}.\\
\end{equation}


The rest of the section is devoted to proving Theorem~\ref{thm:addfixpt}. Observe first that \eqref{eq:addfixpointsProb} is a simple restatement of \eqref{eq:addfixpoints} via~\eqref{eq:probtoenumsig} (using the fact that $\#\mC_{\la'}={n+r\choose n}\#\mC_{\la}$). 
Thus it only remains to prove \eqref{eq:addfixpoints}, which amounts to enumerating $\mS_{\la'}^\al$. For this purpose, we will first define a mapping $\Psi$ from $\mS_{\la'}^\al$ to $\hmS_\la^\al$, where $\hmS_\la^\al$ is a set closely related to $\mS_{\la}^\al$. We shall then count the number of preimages of each element in $\hmS_\la^\al$ under the mapping $\Psi$. Roughly speaking, if $(\pi',A)$ is in $\mS_{\la'}^\al$ and the tuple $A=(A_1,\ldots,A_k)$ is thought as ``marking'' some elements in the cycles of the permutation $\om=\pi'\circ (1,2,\ldots,n+r)$, then the mapping $\Psi$ simply consists in removing all the fixed points of $\pi'$ from the cycle structure of $\om$ and transferring their ``marks'' to the element preceding them in the cycle structure of $\om$.

We introduce some notation. A \emph{multisubset} of $[n]$ is a function $M$ which associates to each integer $i\in[n]$ its \emph{multiplicity} $M(i)$ which is a nonnegative integer. The integer $i$ is said to be \emph{in} the multisubset $M$ if $M(i)>0$. 
 The \emph{size} of $M$ is the sum of multiplicities $\sum_{i=1}^nM(i)$. For a composition $\al=(\al_1,\ldots,\al_k)$, we denote by $\hat{\mA_n^\al}$ the set of tuples $(M_1,\ldots,M_k)$ of \emph{disjoint} multisubsets of $[n]$ (i.e., no element $i\in[n]$ is in more than one multisubset) such that the multisubset $M_j$ has size $\al_j$ for all $j\in[k]$. For $M=(M_1,\ldots,M_k)$ in $\hmA_n^\al$ we say that a permutation $\pi$ of $[n]$ is $M$-separated if no cycle of $\pi$ contains elements of more than one of the multisubsets $M_j$.
Lastly, for a partition $\la$ of $n$ we denote by $\hmS_\la^\al$ the set of pairs $(\pi,M)$ where $\pi$ is a permutation in $\mC_\la$, and $M$ is a tuple in $\hmA_n^\al$ such that the product $\pi\circ (1,2,\ldots,n)$ is $M$-separated.

We now set $\la,\la',\al,k,m,n,r$ to be as in Theorem~\ref{thm:addfixpt}, and define a mapping $\Psi$ from $\mS_{\la'}^\al$ to $\hmS_\la^\al$. 
Let $\pi'$ be a permutation of $[n+r]$ of cycle type $\la'$, and let $e_1<e_2<\cdots<e_n\in[n+r]$ be the elements not fixed by $\pi'$. We denote $\varphi(\pi')$ the permutation $\pi$ defined by setting $\pi(i)=\pi(j)$ if $\pi'(e_i)=e_j$. Observe that $\pi$ has cycle type $\la$.
\begin{rmk}\label{rk:adding-fixed} 
If $e_1<e_2<\cdots<e_n\in[n+r]$ are the elements not fixed by $\pi'$ and $\pi=\varphi(\pi')$, then the cycle structure of the permutation $\pi'\circ (1,2,\ldots,n+r)$ is obtained from the cycle structure of $\pi\circ (1,2,\ldots,n)$ by replacing each element $i\in[n-1]$ by the sequence of elements $F_i=e_i,e_i+1,e_i+2,\ldots,e_{i+1}-1$, and replacing the element $n$ by the sequence of elements $F_n=e_n,e_n+1,e_n+2,\ldots,n+r,1,2,\ldots,e_{1}-1$. In particular, the permutations $\pi\circ (1,2,\ldots,n)$ and $\pi'\circ (1,2,\ldots,n+r)$ have the same number of cycles.
\end{rmk}
Now given a pair $(\pi',A)$ in $\mS_{\la'}^\al$, where $A=(A_1,\ldots,A_k)$, we consider the pair $\Psi(\pi',A)=(\pi,M)$, where $\pi=\varphi(\pi')$ and $M=(M_1,\ldots,M_k)$ is a tuple of multisubsets of $[n]$ defined as follows: for all $j\in [k]$ and all $i\in[n]$ the multiplicity $M_j(i)$ is the number of elements in the sequence $F_i$ belonging to the subset $A_j$ (where the sequence $F_i$ is defined as in Remark~\ref{rk:adding-fixed}). It is easy to see that $\Psi$ is a mapping from $\mS_{\la'}^\al$ to $\hmS_{\la}^\al$.

We are now going to evaluate $\#\mS_{\la'}^\al$ by counting the number of preimages of each element in $\hmS_{\la}^\al$ under the mapping $\Psi$. 
As we will see now, the number of preimages of a pair $(\pi,M)$ in $\hmS_{\la}^\al$ only depends on $M$.
\begin{lemma}\label{lem:transfer}
Let $(\pi,M)\in\hmS_{\la}^\al$, where $M=(M_1,\ldots,M_k)$. Let $s$ be the number of distinct elements appearing in the multisets $M_1,\ldots,M_k$, and let $x=\sum_{j=1}^kM_j(n)$ be the multiplicity of the integer~$n$. Then the number of preimages of the pair $(\pi,M)$ under the mapping $\Psi$ is 
\begin{equation}\label{eq:transfer}
\#\Psi^{-1}(\pi,M)= 
\left\{\begin{array}{ll}
\displaystyle {n+r+s \choose n+m} & \textrm{ if $x=0$,}\\[10pt]
\displaystyle x\,{n+r+s \choose n+m}+ {n+r+s-1 \choose n+m}&\textrm{ otherwise}. 
\end{array}\right.
\end{equation}
\end{lemma}

\begin{proof}
We adopt the notation of Remark~\ref{rk:adding-fixed}, and for all $i\in[n]$ we denote $M_*(i)=\sum_{j=1}^kM_j(i)$ the multiplicity of the integer $i$. In order to construct a preimage $(\pi',A)$ of $(\pi,M)$, where $A=(A_1,\ldots,A_k)$, one has to 
\begin{compactitem}
\item[(i)] choose for all $i\in[n]$ the length $f_i>0$ of the sequence $F_i$ (with $\sum_{i=1}^nf_i=n+r$), 
\item[(ii)] choose the position $b\in[f_n]$ corresponding to the integer $n+r$ in the sequence $F_n$, 
\item[(iii)] if $M_j(i)>0$ for some $i\in[n]$ and $j\in[k]$, then choose which $M_j(i)$ elements in the sequence $F_i$ are in the subset $A_j$.
\end{compactitem}
Indeed, the choices (i), (ii) determine the permutation $\pi'\in\mC_{\la'}$ (since they determine the fixed-points of $\pi'$, which is enough to recover $\pi'$ from $\pi$), while by Remark~\ref{rk:adding-fixed} the choice (iii) determines the tuple of subsets $A=(A_1,\ldots,A_k)$. 

We will now count the number ways of making the choices (i), (ii), (iii) by encoding such choices as rows of (marked and unmarked) boxes as illustrated in Figure~\ref{fig:adding-fixed-bij}. We treat separately the cases $x=0$ and $x\neq 0$. Suppose first $x=0$. To each $i\in[n]$ we associate a row of boxes $R_i$ encoding the choices (i), (ii), (iii) as follows:
\begin{compactitem}
\item[(1)] if $i\neq n$ and $M_*(i)=0$, then the row $R_i$ is made of $f_i$ boxes, the first of which is marked,
\item[(2)] if $i\neq n$ and $M_*(i)>0$, then the row $R_i$ is made of $f_i+1$ boxes, with the first box being marked and $M_*(i)$ other boxes being marked (the marks represent the choice (iii)),
\item[(3)] the row $R_n$ is made of $f_n+1$ boxes, with the first box being marked and an additional box being marked and called \emph{special marked box} (this box represents the choice (ii)).
\end{compactitem}
There is no loss of information in concatenating the rows $R_1,R_2,\ldots,R_n$ given that $M$ is known (indeed the row $R_i$ starts at the $(i+N_i)$th marked box, where $N_i=\sum_{h<i}M_*(h)$\,). This concatenation results in a row of $n+r+s+1$ boxes with $n+m+1$ marks such that the first box is marked and the last mark is ``special''; see Figure~\ref{fig:adding-fixed-bij}. Moreover there are ${n+r+s \choose n+m}$ such rows of boxes and any of them can be obtained for some choices of (i), (ii), (iii). This proves the case $x=0$ of Lemma~\ref{lem:transfer}.\\

\fig{width=\linewidth}{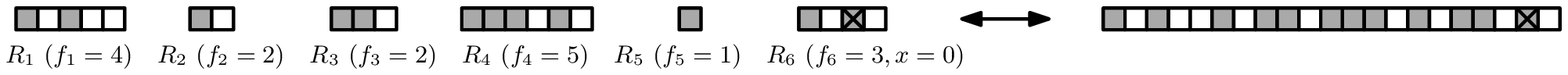}{Example of choices (1),(2),(3) encoded by a sequence of boxes, some of which being marked (indicated in gray), with one mark being special (indicated with a cross). Here $n=6$, $k=2$, $r=11$, $x=0$ and the multisubsets $M_1,M_2$ are defined by $M_1(1)=1$, $M_2(3)=1$, $M_1(4)=3$, and $M_j(i)=0$ for the other values of $i,j$.}

We now suppose $x>0$. We reason similarly as above but there are now two possibilities for the row $R_n$, depending on whether or not the integer $n+r$ belongs to one of the subsets $A_1,\ldots,A_k$. In order to encode a preimage such that $n+r$ belong to one of the subsets $A_1,\ldots,A_k$ the condition $(3)$ above must be changed to 
\begin{compactitem}
\item[(3')] the row $R_n$ is made of $f_n+1$ boxes, with the first box being marked and $x$ other boxes being marked, one of which being called \emph{special marked box}.
\end{compactitem}
In this case, concatenating the rows $R_1,R_2,\ldots,R_n$ gives a row of $n+r+s$ boxes with $n+m$ marks, with the first box being marked and one of the $x$ last marked boxes being special. There are $x {n+r+s-1 \choose n+m-1}$ such rows and each of them comes from a unique choice of (i), (ii) and (iii). 

Lastly, in order to encode a preimage such that $n+r$ does not belong to one of the subsets $A_1,\ldots,A_k$ the condition $(3)$ above must be changed to 
\begin{compactitem}
\item[(3'')] the row $R_n$ is made of $f_n+1$ boxes, with the first box being marked and $x+1$ other boxes being marked, one of which being called \emph{special marked box}.
\end{compactitem}
In this case, concatenating the rows $R_1,R_2,\ldots,R_n$ gives a row of $n+r+s$ boxes with $n+m+1$ marks, with the first box being marked and one of the $x+1$ last marked boxes being special. There are $(x+1) {n+r+s-1 \choose n+m}$ such rows and each of them comes from a unique choice of (i), (ii) and (iii). 

Thus, in the case $x>0$ one has
$$\#\Psi^{-1}(\pi,M)=x {n+r+s-1 \choose n+m-1}+(x+1) {n+r+s-1 \choose n+m}=x\,{n+r+s \choose n+m}+ {n+r+s-1 \choose n+m}.$$
This completes the proof of Lemma \ref{lem:transfer}.
\end{proof}

We now complete the proof of Theorem~\ref{thm:addfixpt}. For any composition $\ga=(\ga_1,\ldots,\ga_k)$, we denote by $\hmS_\la^{\al,\ga}$ the set of pairs $(\pi,M)$ in $\hmS_\la^\al$, where the tuple $M=(M_1,\ldots,M_k)$ is such that for all $j\in[k]$ the multisubset $M_j$ (which is of size $\al_j$) contains exactly $\ga_j$ distinct elements. 
Summing~\eqref{eq:transfer} gives
\begin{equation}\label{eq:sum-transfer}
\sum_{(\pi,M)\in\hmS_\la^{\al,\ga}}\!\!\!\#\Psi^{-1}(\pi,M)=\left((\EE(X)+\PP(X=0)){n+r+|\ga| \choose n+m}+\PP(X>0){n+r+|\ga|-1 \choose n+m}\right)\,\#\hmS_\la^{\al,\ga},
\end{equation}
where $X$ is the random variable defined as $X=\sum_{j=1}^kM_j(n)$ for a pair $(\pi,M)$ chosen uniformly randomly in $\hmS_\la^{\al,\ga}$, $\EE(X)$ is the expectation of this random variable, and $\PP(X>0)=1-\PP(X=0)$ is the probability that $X$ is positive. 
\begin{lemma}\label{lem:cyclic}
With the above notation, $\displaystyle \EE(X)=\frac{m}{n}$, and $\displaystyle \PP(X>0)=\frac{|\ga|}{n}$.
\end{lemma}
\begin{proof}
The proof is simply based on a cyclic symmetry. For $i\in[n]$ we consider the random variable $X_i=\sum_{j=1}^kM_j(i)$ for a pair $(\pi,M)$ chosen uniformly randomly in $\hmS_\la^{\al,\ga}$. It is easy to see that all the variables $X_1,\ldots,X_n=X$ are identically distributed since the set $\hmS_\la^{\al,\ga}$ is unchanged by cyclically shifting the value of the integers $1,2,\ldots,n$ in pairs $(\pi,M)\in\hmS_\la^{\al,\ga}$. Therefore, 
$$n\,\EE(X)=\sum_{i=1}^n\EE(X_i)=\EE\left(\sum_{i=1}^nX_i\right)=\EE(m)=m,$$
and 
$$n\,\PP(X>0)=\sum_{i=1}^n\PP(X_i>0)=\EE\left(\sum_{i=1}^n1_{X_i>0}\right)=\EE\left(|\ga|\right)=|\ga|.$$
\end{proof}

We now enumerate the set $\hmS_\la^{\al,\ga}$.
Observe that any pair $(\pi,M)$ in $\hmS_\la^{\al,\ga}$ can be obtained (in a unique way) from a pair $(\pi,A)$ in $\mS_\la^{\ga}$ by transforming  $A=(A_1,\ldots,A_k)$ into  $M=(M_1,\ldots,M_k)$ as follows: for each $j\in[k]$ one has to assign a positive multiplicity $M_j(i)$ for all $i\in A_j$ so as to get a multisubset $M_j$ of size $\al_j$. There are $\binom{\al_j-1}{\ga_j-1}$ ways of performing the latter task, hence 
$$\#\hmS_\la^{\al,\ga}=\prod_{i=1}^k \binom{\al_i-1}{\ga_i-1}\,\#\mS_\la^\ga.$$
Using this result and Lemma~\ref{lem:cyclic} in \eqref{eq:sum-transfer} gives
\begin{equation}\label{eq:sum-transfer2}\nonumber
\sum_{(\pi,M)\in\hmS_\la^{\al,\ga}}\#\Psi^{-1}(\pi,M)=\left(\frac{m+n-|\ga|}{n}{n+r+|\ga| \choose n+m}+\frac{|\ga|}{n}{n+r+|\ga|-1 \choose n+m}\right)\prod_{i=1}^k \binom{\al_i-1}{\ga_i-1}\,\#\mS_\la^\ga.
\end{equation}
Observe that the above expression is 0 unless $\ga$ is less or equal to $\al$ componentwise. Finally, one gets
\begin{equation}\label{eq:fixedpointseq2}
\#\mS^{\alpha}_{\la'} = \sum_{\gamma\leq\al,~ \ell(\gamma)=k} \left(\frac{m+n-|\ga|}{n}{n+r+|\ga| \choose n+m}+\frac{|\ga|}{n}{n+r+|\ga|-1 \choose n+m}\right)\prod_{i=1}^k \binom{\al_i-1}{\ga_i-1}\,\#\mS_\la^\ga,
\end{equation}
where the sum is over compositions $\ga$ with $k$ parts, which are less or equal to $\al$ componentwise. Lastly, by Corollary~\ref{cor:sepsym}, the cardinality $\#\mS^{\gamma}_{\la'}$ only depends on the composition $\al$ through the length and size of $\al$. Therefore, one can use \eqref{eq:fixedpointseq2} with $\al=(m-k+1,1^{k-1})$, in which case the compositions $\ga$ appearing in the sum are  of the form  $\ga=(m-k-p+1,1^{k-1})$ for some $p\leq m-k$. This gives \eqref{eq:addfixpoints} and completes the proof of Theorem~\ref{thm:addfixpt}.\hfill$\square$



\section{Bijective proofs and interpretation in terms of maps}\label{sec:maps}
In this section we explain how certain results of this paper can be interpreted in terms of \emph{maps}, and can be proved bijectively. In particular, we shall interpret the sets $\mT_{\ga,\de}^\al$ of ``separated colored factorizations'' (defined in Section~\ref{sec:strategy}) in terms of maps. We can then extend a bijection from \cite{OB:Harer-Zagier-non-orientable} in order to prove bijectively the symmetry property stated in Corollary~\ref{cor:sepsym}.

\subsection{Interpretations of (separated) colored factorizations in terms of maps}
We first recall some definitions about maps. Our \emph{graphs} are undirected, and they can have multiple edges and loops.  Our \emph{surfaces} are two-dimensional, compact, boundaryless, orientable, and considered up to homeomorphism; such a surface is characterized by its genus. A connected graph is \emph{cellularly embedded} in a surface if its edges are not crossing and its \emph{faces} (connected components of the complement of the graph) are simply connected. A \emph{map} is a cellular embedding of a connected graph in an orientable surface considered up to homeomorphism. A map is represented in Figure \ref{fig:one-face-maps}. 
By cutting an edge in its midpoint one gets two \emph{half-edges}.  A map is \emph{rooted} if one of its half-edges is distinguished as the \emph{root}. In what follows we shall consider rooted bipartite maps, and consider the unique proper coloring of the vertices in black and white such that the root half-edge is incident to a black vertex.

\fig{width=.6\linewidth}{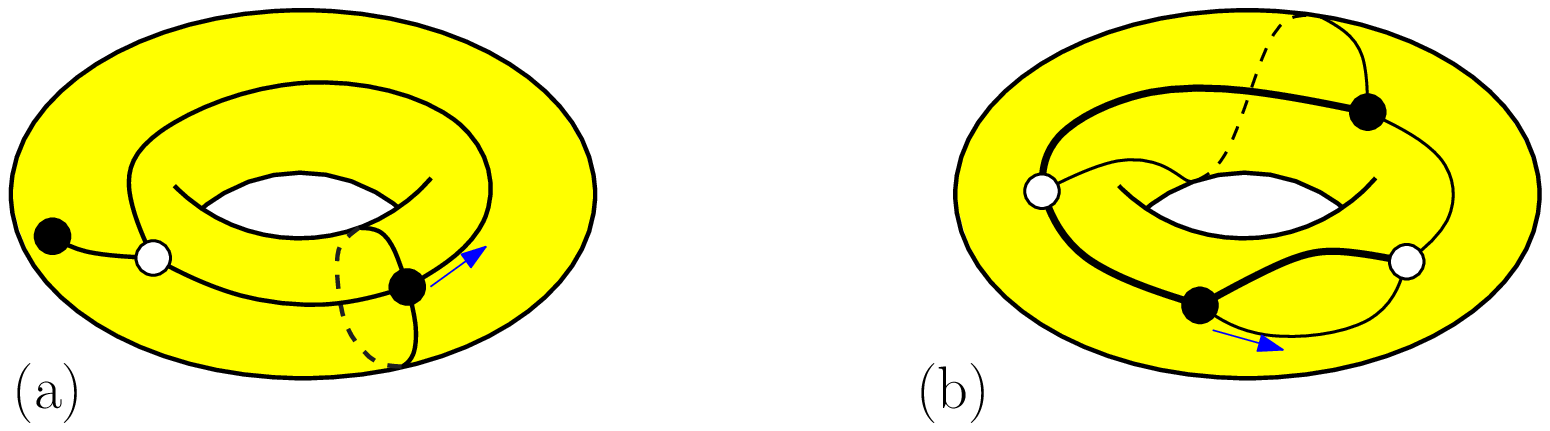}{(a) A rooted bipartite one-face map. (b) A rooted bipartite tree-rooted map (the spanning tree is indicated by thick lines). The root half-edge is indicated by an arrow.}

By a classical encoding (see e.g. \cite{LZ}), for any partitions $\la,\mu$ of $n$, the solutions $(\pi_1,\pi_2)\in\mC_\la\times\mC_\mu$ of the equation $\pi_1\circ\pi_2=(1,2,\ldots,n)$ are in bijection with the rooted one-face bipartite maps such that black and white vertices have degrees given by the permutations $\la$ and $\mu$ respectively.  That is, the number of black (resp. white) vertices of degree $i$ is equal to the number of parts of the partition $\lambda$ (resp. $\mu$) equal to $i$. Let  $\ga=(\ga_1,\ldots,\ga_\ell)$, $\de=(\de_1,\ldots,\de_{\ell'})$ be compositions of $n$ and let $\al=(\al_1,\ldots,\al_k)$ be a composition of $m\leq n$. A rooted bipartite map is $(\ga,\de)$-colored if  its black vertices  are colored in $[\ell]$ (that is, every vertex is assigned a ``color'' in $[\ell]$) in such a way that $\ga_i$ edges are incident to black vertices of color $i$, and its white vertices are colored in $[\ell']$ in such a way that $\de_i$ edges are incident to white vertices of color $i$. Through the above mentioned encoding, the set $\mB_{\ga,\de}$ of colored factorizations of the $n$-cycles defined in Section~\ref{sec:strategy} corresponds to the set  of $(\ga,\de)$-colored rooted bipartite one-face maps. Similarly, the sets $\mT_{\ga,\de}^\al$ of ``separated colored factorizations'' corresponds to the set of  $(\ga,\de)$-colored rooted bipartite one-face maps with some marked edges, such that for all $i\in[k]$ exactly $\al_i$ marked edges are incident to white vertices colored $i$. 

The results in this paper can then be interpreted in terms of maps.
For instance, one can interpret \eqref{eq:GF-explicit} in the case  $m=k=0$ (no marked edges) as follows:  
$$\sum_{\la\vdash n}\sum_{M\in\mB_\la} \!p_\la(\xx)\,t^{\#\textrm{white vertices}}=\GG_n^\emptyset(\xx,t) =\sum_{r=1}^{n}\sum_{\la \vdash n,~ \ell(\la)\leq n-r+1} \! \!m_{\la}(\xx)\binom{t}{r} \frac{n(n-\ell(\la))!(n-r)!}{(n-r-\ell(\la)+1)!} \binom{n-1}{n-r},$$
where $\mB_\la$ is the set of rooted bipartite one-face maps such that black vertices  have degrees given by the partition $\la$. 
The results in Subsection \ref{subsec:involution} can also be interpreted in terms of \emph{general} (i.e., non-necessarily bipartite) maps. Indeed, the set $\mM_N=\mB_{2^N}$ can be interpreted as the set of general rooted one-face maps with $N$ edges (because a bipartite map in which every black vertex has degree two can be interpreted as a general map upon contracting the black vertices).
Therefore one can interpret \eqref{eq:GF-Mapsexplicit} in the case  $m=k=0$ (no marked edges) as follows:  
\begin{equation}\label{eq:HZ}
\sum_{M\in \mM_N}t^{\#\textrm{vertices}}=H_{N}^{\emptyset}(t) = N\sum_{r=1}^{N+1} \binom{t}{r} 2^{r-N} \frac{(2N-r)!}{(N-r+1)!}\binom{2N-1}{2N-r}.
\end{equation}
This equation is exactly the celebrated Harer-Zagier formula \cite{HZ}.

\subsection{Bijection for separated colored factorizations, and symmetry}
In this section, we explain how some of our proofs could be made bijective. In particular we will use bijective results obtained in \cite{OB:Harer-Zagier-non-orientable} in order to prove the symmetry result stated in Corollary~\ref{cor:sepsym}.\\

We first recall the bijection obtained in \cite{OB:Harer-Zagier-non-orientable} about the sets $\mB_{\ga,\de}$.
We define a \emph{tree-rooted map} to be a rooted map with a marked spanning tree; see Figure~\ref{fig:one-face-maps}(b). 
We say that a bipartite tree-rooted map is $(\ell,\ell')$\emph{-labelled} if it has $\ell$ black vertices labelled with distinct labels in $[\ell]$, and $\ell'$ white vertices labelled with distinct labels in $[\ell']$. It was shown in \cite{OB:Harer-Zagier-non-orientable} that for any compositions $\ga=(\ga_1,\ldots,\ga_\ell)$, $\de=(\de_1,\ldots,\de_{\ell'})$ of $n$, the set $\mB_{\ga,\de}$ 
is in bijection with the set of $(\ell,\ell')$-labelled bipartite tree-rooted maps such that the black (resp. white) vertex labelled $i$ has degree $\ga_i$ (resp. $\de_i$). 

From this bijection, it is not too hard to derive the enumerative formula \eqref{eq:colored-factorizations} (see Remark~\ref{rmk:symmetry}).
We now adapt the bijection established in \cite{OB:Harer-Zagier-non-orientable} to the sets $\mT_{\ga,\de}^\al$ of ``separated colored factorizations''.
For a composition $\al=(\al_1,\ldots,\al_k)$, a $(\ell,\ell')$-labelled bipartite maps is said to be \emph{$\al$-marked} if $\al_i$ edges incident to the white vertex labelled $i$ are marked for all $i$ in $[k]$.
\begin{theorem}\label{thm:bij}
The bijection in \cite{OB:Harer-Zagier-non-orientable} extends into a bijection between the set $\mT_{\ga,\de}^\al$ and the set of $\al$-marked $(\ell,\ell')$-labelled bipartite tree-rooted maps with $n$ edges such that the black (resp. white) vertex labelled $i$ has degree $\ga_i$ (resp. $\de_i$).
\end{theorem}

We will now show that the bijection given by Theorem \ref{thm:bij} easily implies 
\begin{equation}\label{eq:sym-sep-colored}
\#\mT_\ga^\al(r)=\#\mT_\ga^\be(r),
\end{equation}
whenever the compositions $\al$ and $\be$ have the same length and size. Observe that, in turn, \eqref{eq:sym-sep-colored} readily implies Corollary~\ref{cor:sepsym}.

By Theorem~\ref{thm:bij}, the set $\mT_\ga^\al(r)$ specified by Definition~\ref{def:setT} is in bijection with the set $\mTR_\ga^\al(r)$ of $\al$-marked $(\ell,k+r)$-labelled bipartite tree-rooted maps with $n$ edges such that the black vertex labelled $i$ has degree $\ga_i$. We will now describe a bijection between the sets $\mTR_\ga^\al(r)$ and $\mTR_\ga^\be(r)$ when $\al$ and $\be$ have the same length and size. For this purpose it is convenient to interpret maps as graphs endowed with a rotation system.   A \emph{rotation system} of a graph $G$ is an assignment for each vertex $v$ of $G$ of a cyclic ordering of the half-edges incident to $v$. Any map $M$ defines a rotation system $\rho(M)$ of the underlying graph: the cyclic orderings are given by the clockwise order of the half-edges around the vertices. This correspondence is in fact bijective (see e.g. \cite{MT}): for any connected graph $G$ the mapping $\rho$ gives a bijection between maps having underlying graph $G$ and the rotation systems of $G$. Using the ``rotation system'' interpretation, any map can be represented  in the plane (with edges allowed to cross each other) by choosing the clockwise order of the half-edges around each vertex to represent the rotation system; this is the convention used in Figures \ref{fig:marked-tree-rooted} and \ref{fig:easy-tree-rooted}.

\fig{width=.8\linewidth}{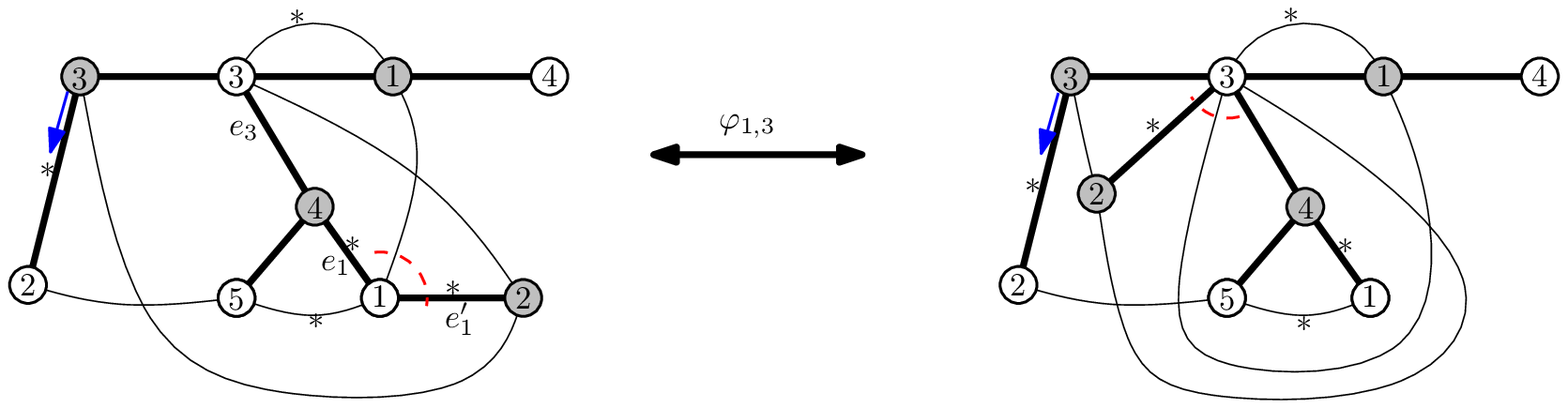}{Left: a $(3,1,1)$-marked $(4,5)$-labelled bipartite tree-rooted map. Right: the $(2,1,2)$-marked $(4,5)$-labelled bipartite tree-rooted map obtained by applying the mapping $\varphi_{1,3}$. In this figure, maps are represented using the ``rotation system interpretation'', so that the edge-crossings are irrelevant. The spanning trees are drawn in thick lines, the marked edges are indicated by stars, and the root half-edge is indicated by an arrow.}

We now prove \eqref{eq:sym-sep-colored} it is  sufficient to establish a bijection between the sets $\mTR_\ga^\al(r)$ and $\mTR_\ga^\be(r)$ in the case $\al=(\al_1,\ldots,\al_k)$, $\be=(\be_1,\ldots,\be_k)$ with $\be_i=\al_i-1$, $\be_j=\al_j+1$ and $\al_s=\be_s$ for $s\neq i,j$. Let $M$ be an $\al$-marked $(\ell,\ell')$-labelled bipartite tree-rooted map. We consider the path joining the white vertices $i$ and $j$ in the spanning tree of $M$. Let $e_i$ and $e_j$ be the edges of this path incident to the white vertices $i$ and $j$ respectively; see Figure~\ref{fig:marked-tree-rooted}. We consider the first marked edge $e_i'$ following $e_i$ in clockwise order around the vertex $i$ (note that $e_i\neq e_i'$ since $\al_i=\be_i+1>1$). We then define $\varphi_{i,j}(M)$ as the map obtained by ungluing from the vertex $i$ the half-edge of $e_i'$ as well as all the half-edges appearing strictly between $e_i$ and $e_i'$, and gluing them (in the same clockwise order) in the corner following $e_j$ clockwise around the vertex $j$. Figure~\ref{fig:marked-tree-rooted} illustrates the mapping $\varphi_{1,3}$. It is easy to see that $\varphi_{i,j}(M)$ is a tree-rooted map, and that $\varphi_{i,j}$ and $\varphi_{j,i}$ are reverse mappings. Therefore $\varphi_{i,j}(M)$ is a bijection between $\mTR_\ga^\al(r)$ and $\mTR_\ga^\be(r)$. This proves \eqref{eq:sym-sep-colored}.

\begin{rmk} \label{rmk:symmetry} 
By an argument similar to the one used above to prove \eqref{eq:sym-sep-colored}, one can prove that if $\ga,\ga',\de,\de'$ are compositions of $n$ such that $\ell(\ga)=\ell(\ga')$ and $\ell(\de)=\ell(\de')$ then $\mB_{\ga,\de}=\mB_{\ga',\de'}$ (this is actually done in a more general setting in \cite{BM1}). From this property one can compute the cardinality of $\mB_{\ga,\de}$ by choosing the most convenient compositions $\ga$, $\de$ of length $\ell$ and $\ell'$. We take $\ga=(n-\ell+1,1,1,\ldots,1)$ and $\de=(n-\ell'+1,1,1,\ldots,1)$, so that $\#\mB_{\ga,\de}$ is the number of $(\ell,\ell')$-labelled bipartite tree-rooted maps with the black and white vertices labelled 1 of degrees $n-\ell+1$ and $n-\ell'+1$ respectively, and all the other vertices of degree 1. In order to construct such an object (see Figure~\ref{fig:easy-tree-rooted}), one must choose the unrooted plane tree (1 choice), the labelling of the vertices ($(\ell-1)!(\ell'-1)!$ choices), the $n-\ell-\ell'+1$ edges not in the tree (${n-\ell \choose n-\ell-\ell'+1}{n-\ell' \choose n-\ell-\ell'+1}(n-\ell'-\ell'+1)!$ choices), and lastly the root ($n$ choices). This gives \eqref{eq:colored-factorizations}.
\end{rmk}

\fig{width=.3\linewidth}{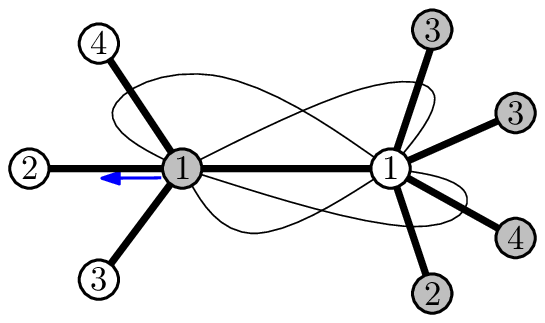}{A tree-rooted map in $\mB_{\ga,\de}$, where $\ga=(8,1,1,1,1)$, $\de=(9,1,1,1)$. Here the map is represented using the ``rotation system interpretation'', so that the edge-crossings are irrelevant.}

\subsection{A direct proof of Theorem~\ref{thm:GF-Maps}}
In Section \ref{sec:results} we obtained Theorem~\ref{thm:GF-Maps} as a consequence of Theorem~\ref{thm:GF}. Here we explain how to obtain it directly.

First of all, by a reasoning identical to the one used to derive~\eqref{eq:colsepmap} one gets 
\begin{equation}\label{eq:proof-GF-Maps}
H_{N}^{\alpha}(t+k) = \sum_{r=0}^{2N-m} \binom{t}{r}\,\# \mathcal{U}^{\alpha}(r),
\end{equation}
where $\mathcal{U}^{\alpha}(r)$ is the set of triples $(\pi,A,c_2)$ where $\pi$ is a fixed-point free involution of $[2N]$, $A$ is in $\mA^{\alpha}_n$ and $c_2$ is a a cycle coloring of the product $\pi \circ (1,2,\ldots,2N)$ in $[k+r]$ such that every color in $[k+r]$ is used and for all $i$ in $[k]$ the elements in the subset $A_i$ are colored $i$. 

In order to enumerate $\mathcal{U}^{\alpha}(r)$ one considers for each composition $\ga=(\ga_1,\ldots,\ga_\ell)$ the set $\mathcal{M}_{\gamma}$ of pairs $(\pi,c_2)$, where $\pi$ is a fixed-point-free involution of $[2N]$ and $c_2$ is a cycle coloring of the permutation $\pi \circ (1,2,\ldots,2N)$ such that $\gamma_i$ elements are colored $i$ for all $i\in [\ell]$. 
One then uses the following analogue of~\eqref{eq:colored-factorizations}: 
\begin{equation}\label{eq:HZrefined}
\#\mathcal{M}_{\gamma} = \frac{N(2N-\ell)!}{(N-\ell+1)!} 2^{\ell-N}.
\end{equation}
Using this result in conjunction with Lemma~\ref{lem:nb-markings}, one then obtains the following analogue of~\eqref{eq:cardT}:
\begin{equation}\nonumber 
\#\mathcal{U}^{\alpha}(r) = \frac{N(2N-k-r)!}{(N-k-r+1)!}\binom{2N+k-1}{2N-m-r}.
\end{equation}
Plugging this result in \eqref{eq:proof-GF-Maps} completes the proof of Theorem \ref{thm:GF-Maps}.\\

Similarly as \eqref{eq:colored-factorizations}, Equation \eqref{eq:HZrefined} can be obtained bijectively. Indeed by a classical encoding, the set $\mathcal{M}_{\gamma}$ is in bijection with the set of rooted one-face maps with vertices colored  in $[\ell]$ in such a way that for all $i\in[\ell]$, there are exactly $\ga_i$ half-edges incident to vertices of color $i$. Using this interpretation, it was proved in~\cite{OB:Harer-Zagier-non-orientable} that the set $\mathcal{M}_{\ga}$  is in bijection with the set of tree-rooted maps with $\ell$ vertices labelled with distinct labels in $[\ell]$ such that the vertex labelled $i$ has degree $\gamma_i$. The latter set is easy to enumerate (using symmetry as in Remark \ref{rmk:symmetry}) and one gets \eqref{eq:HZrefined}.

\section{Concluding remarks: strong separation and connection coefficients}\label{sec:conclusion}
Given a tuple $A=(A_1,\ldots,A_k)$ of disjoint subsets of $[n]$, a permutation $\pi$ is said to be \emph{strongly $A$-separated} if  each of the subsets $A_i$, for $i\in[k]$ is included in a distinct cycle of $\pi$. Given a partition $\la$ of $n$ and a composition $\al$ of $m\leq n$, we denote by $\ppi_{\la}^\al$ the probability that the product $\om\circ \rho$ is \emph{strongly} $A$-separated, where $\om$ (resp. $\rho$) is a uniformly random permutation of cycle type $\la$ (resp. $(n)$) and $A$ is a fixed tuple in $\mA_n^\al$. In particular, for a composition $\al$ of size $m=n$, one gets 
$$\ppi_{\la}^\al=\frac{K_{\la,(n)}^\al\prod_{i=1}^k (\alpha_i-1)!}{(n-1)!\,\,\#\mC_\la},$$ 
where $K_{\la,(n)}^\al$ is the  \emph{connection coefficient of the symmetric group} counting the number of solutions $(\om,\rho)\in \mC_\la\times \mC_{(n)}$, of the equation $\om\circ \rho=\phi$ where $\phi$ is a fixed permutation of cycle type $\al$. 

We now argue that the separation probabilities $\{\psig_{\la}^\al\}_{\al \models m}$ computed in this paper are enough to determine the probabilities $\{\ppi_{\la}^\al\}_{\al \models m}$. Indeed, it is easy to prove that 
\begin{equation}\label{eq:weak-to-strong}
\psig_{\la}^\al = \sum_{\be \preceq \al} R_{\al,\be}\ppi_{\la}^\be,
\end{equation}
where the sum is over the compositions $\be=(\be_1,\ldots,\be_\ell)$ of size $m=|\al|$ such that there exists $0=j_0<j_1<j_2<\cdots<j_k=\ell$ such that  $(\be_{j_{i-1}+1},\be_{j_{i-1}+1},\ldots,\be_{j_{i}})$ is a composition of $\al_i$ for all $i\in[k]$, and $R_{\al,\be}=\prod_{i=1}^kR_i$ where $R_i$ is the number of ways of partitioning a set of size $\al_i$ into blocks of respective sizes $\be_{j_{i-1}+1},\be_{j_{i-1}+1},\ldots,\be_{j_{i}}$.
Moreover, the matrix $(R_{\al,\be})_{\al,\be\models m}$ is invertible (since the matrix is upper triangular for the lexicographic ordering of compositions). Thus, from the separation probabilities $\{\psig_{\la}^\al\}_{\al \models m}$ one can deduce the strong separation probabilities $\{\ppi_{\la}^\al\}_{\al \models m}$ and in particular, for $m=n$, the connection coefficients $K_{\la,(n)}^\al$ of the symmetric group.\\

\noindent \textbf{Acknowledgment:} We thank Taedong Yun for several stimulating discussions.

\bibliography{biblio-separation}{}

\begin{thebibliography}{10}

\bibitem{OB:Harer-Zagier-non-orientable}
O.~Bernardi.
\newblock An analogue of the {H}arer-{Z}agier formula for unicellular maps on
  general surfaces.
\newblock {\em Adv. in Appl. Math.}, 48(1):164--180, 2012.

\bibitem{BM1}
O.~Bernardi and A.H. Morales.
\newblock Bijections and symmetries for the factorizations of the long cycle.
\newblock Submitted, \href{http://arxiv.org/abs/1112.4970}{ArXiv:1112.4970},
  2011.

\bibitem{MBRF}
M.~B\'ona and R.~Flynn.
\newblock The average number of block interchanges needed to sort a permutation
  and a recent result of stanley.
\newblock {\em Inform. Process. Lett.}, 109(16):927--931, 2009.

\bibitem{HZ}
J.~Harer and D.~Zagier.
\newblock The {E}uler characteristic of the moduli space of curves.
\newblock {\em Invent. Math.}, 85(3):457--485, 1986.

\bibitem{J}
D.M. Jackson.
\newblock Some combinatorial problems associated with products of conjugacy
  classes of the symmetric group.
\newblock {\em J. Combin. Theory Ser. A}, 49(2), 1988.

\bibitem{LZ}
S.K. Lando and A.K. Zvonkin.
\newblock {\em Graphs on surfaces and their applications}.
\newblock Springer-Verlag, 2004.

\bibitem{MT}
B.~Mohar and C.~Thomassen.
\newblock {\em Graphs on surfaces}.
\newblock J. Hopkins Univ. Press, 2001.

\bibitem{MV}
A.H. Morales and E.A. Vassilieva.
\newblock Bijective enumeration of bicolored maps of given vertex degree
  distribution.
\newblock In {\em DMTCS Proceedings, 21st International Conference on Formal
  Power Series and Algebraic Combinatorics (FPSAC 2009)}.

\bibitem{SV}
G.~Schaeffer and E.A. Vassilieva.
\newblock A bijective proof of {J}ackson's formula for the number of
  factorizations of a cycle.
\newblock {\em J. Combin. Theory, Ser. A}, 115(6):903--924, 2008.

\bibitem{EC2}
R.P. Stanley.
\newblock {\em Enumerative combinatorics, volume 2}.
\newblock Cambridge University Press, 1999.

\bibitem{DS}
R.P. Stanley.
\newblock Products of cycles, 2010.
\newblock Slides for the conference \emph{Permutation Patterns 2010}, {\em
  http://www-math.mit.edu/$\sim$rstan/transparencies/cycleprod.pdf}.

\bibitem{EV}
E.A. Vassilieva.
\newblock Explicit monomial expansions of the generating series for connection
  coefficients.
\newblock \href{http://arxiv.org/abs/1111.6215}{ArXiv:1111.6215}, 2011.

\end{thebibliography}
\bibliographystyle{plain}

\noindent Olivier Bernardi\\
Department of Mathematics, Massachusetts Institute of Technology; Cambridge, MA USA 02139\\
{\tt bernardi@math.mit.edu}

\medskip

\noindent Rosena R. X. Du\\
Department of Mathematics, East China Normal University; Shanghai, China 200041\\
{\tt rxdu@math.ecnu.edu.cn}

\medskip

\noindent Alejandro H. Morales\\
Department of Mathematics, Massachusetts Institute of Technology; Cambridge, MA USA 02139\\
{\tt ahmorales@math.mit.edu}

\medskip

\noindent Richard P. Stanley\\
Department of Mathematics, Massachusetts Institute of Technology; Cambridge, MA USA 02139\\
{\tt rstan@math.mit.edu}

\end{document}